\documentclass[leqno]{amsart}
\usepackage{amsmath,amsfonts,amsthm,amssymb,indentfirst,epic,url}

\newcommand{\<}{\kern.0833em}
\newcommand{\strt}[1][1.7]{\vrule width0pt height0pt depth#1pt}

\newtheorem{theorem}{Theorem}
\newtheorem{lemma}[theorem]{Lemma}
\newtheorem{corollary}[theorem]{Corollary}
\newtheorem{proposition}[theorem]{Proposition}
\newtheorem{definition}[theorem]{Definition}
\newtheorem{question}[theorem]{Question}
\newtheorem{example}[theorem]{Example}

\newcommand{\En}{\mathrm{Endo}}
\newcommand{\lm}{\varprojlim}

\begin{document}

\begin{center}
\texttt{Comments, corrections, and related references welcomed,
as always!}\\[.5em]
{\TeX}ed \today
\vspace{2em}
\end{center}

\title%
{Homomorphic images of pro-nilpotent algebras}
\thanks{This preprint is accessible online at
\url{http://math.berkeley.edu/~gbergman/papers/}.
}


\author{George M. Bergman}
\address{University of California\\
Berkeley, CA 94720-3840, USA}
\email{gbergman@math.berkeley.edu}

\begin{abstract}
It is shown that any finite-dimensional homomorphic image of an
inverse limit of nilpotent not-necessarily-associative algebras
over a field is nilpotent.
More generally, this is true of algebras over a general commutative
ring $k,$ with ``finite-dimensional'' replaced by ``of
finite length as a $\!k\!$-module''.

These results are obtained by considering the multiplication algebra
$M(A)$ of an algebra $A$ (the associative algebra of $\!k\!$-linear maps
$A\to A$ generated by left and right multiplications by elements
of $A),$ and its behavior with respect
to nilpotence, inverse limits, and homomorphic images.

As a corollary, it is shown that a finite-dimensional homomorphic
image of an inverse limit of finite-dimensional solvable
Lie algebras over a field of characteristic~$0$ is solvable.

Examples are given showing that {\em infinite}-dimensional
homomorphic images of inverse limits of nilpotent algebras can have
properties far from those of nilpotent algebras; in particular,
properties that imply that they are not residually nilpotent.

Several open questions and directions for further
investigation are noted.
\end{abstract}
\maketitle

\section{General definitions.}\label{S.defs}

Throughout this note, $k$ will be a commutative associative unital ring,
and an ``algebra'' will mean a $\!k\!$-algebra;
i.e., a $\!k\!$-module $A$ given with a $\!k\!$-bilinear
multiplication $A\times A\to A,$ not necessarily associative or unital.

Recall that if $A$ is a nonunital {\em associative} algebra contained
in a {\em unital} associative algebra $A',$ then the identity
\begin{equation}\begin{minipage}[c]{35pc}\label{d.(1+x)(1+y)}
$(1+x)(1+y)\ =\ 1+(x+y+xy)\quad (x,y\in A)$
\end{minipage}\end{equation}
which holds in $A'$ motivates
one to define, on $A,$ the operation of {\em quasimultiplication},
\begin{equation}\begin{minipage}[c]{35pc}\label{d.quasimult}
$x*y\ =\ x+y+xy.$
\end{minipage}\end{equation}
This is again associative, and has $0$ as identity element;
an element $x\in A$ is called {\em quasiinvertible} if
there exists $y\in A$ such that $x*y=y*x=0;$ equivalently,
if $1+x$ is invertible in the multiplicative
submonoid $\{1+u\mid u\in A\}$ of $A'.$
In particular, every nilpotent element $x\in A$ is
quasiinvertible, with quasiinverse $-x+x^2-\dots+(-x)^n+\cdots\,.$
The Jacobson radical of $A$ is the largest ideal consisting of
quasiinvertible elements; so an associative algebra is
Jacobson radical if and only if every element is quasiinvertible.
We shall write ``Jacobson radical'' and ``radical'' interchangeably
in this note, using the former mainly in statements of results.
We shall only use these terms in reference to associative algebras.

If $A$ is a not-necessarily-associative algebra,
let us write $\En(A)$ for the associative unital $\!k\!$-algebra
of all endomorphisms of $A$ as a $\!k\!$-module.
(Since $\mathrm{End}(A)$ should denote the set of algebra
endomorphisms, we use this slightly different symbol for the algebra
of module endomorphisms.)
For every $x\in A,$ we define the left and right multiplication maps
$l_x,\ r_x\in\En(A)$ by
\begin{equation}\begin{minipage}[c]{35pc}\label{d.rl}
$l_x(y)\ =\ x\,y,\quad r_x(y)\ = \ y\,x,$
\end{minipage}\end{equation}
and denote by $M(A)$ the generally {\em nonunital} subalgebra
of $\En(A)$ generated by these maps, as $x$ runs over~$A;$ this is
called the {\em multiplication algebra} of $A$ \cite[p.14]{Schafer}.

An algebra $A$ is called {\em nilpotent} if for some $n>0,$ all
length-$\!n\!$
products of elements of $A,$ no matter how bracketed, are zero.
We shall see that $M(A)$ is nilpotent if and only if $A$ is
nilpotent (not hard to prove, but not quite trivial either).

\section{Preview of the idea of the proof of our main result, and of a counterexample.}\label{S.preview}

If $A$ is a nilpotent algebra, then the associative algebra
$M(A),$ being nilpotent, will in particular be radical.
Now though the property of being nilpotent is not preserved by inverse
limits, that of being radical is, and is
likewise preserved under surjective homomorphisms.
To use these facts, we have to know how $M(A)$
behaves with respect to homomorphisms and inverse limits.

In general, a homomorphism of algebras $h:A\to B$ does not induce a
homomorphism $M(h):M(A)\to M(B);$ but we shall see that
it does if $h$ is surjective, and that $M(h)$ is then also surjective.
The need for $h$ to be surjective will not be a problem for us, because
if an algebra $A$ is an inverse limit of nilpotent algebras $A_i,$
then by replacing the $A_i$ with appropriate subalgebras, we can
get a new system having the same inverse limit $A,$ and such
that the new projection maps $A\to A_i$ and connecting
maps $A_i\to A_j$ are surjective.
Once these conditions hold, we shall find that
\begin{equation}\begin{minipage}[c]{35pc}\label{d.Mlm}
$M(\lm A_i)\ \subseteq\ \lm M(A_i)\ \subseteq\ \En(\lm A_i).$
\end{minipage}\end{equation}
Hence, if the $A_i$ are all nilpotent, the elements
of $M(\lm A_i)$ will all have quasiinverses in the radical
algebra $\lm M(A_i),$ and hence in $\En(\lm A_i).$
From this we shall be able to deduce that if $B$ is a homomorphic
image of $A=\lm A_i,$ then for all $u\in M(B),$ the
linear map $1+u\in\En(B)$ is surjective.

If, moreover, $B$ has finite length as a $\!k\!$-module,
this surjectivity makes these
maps $1+u$ $(u\in M(B))$ invertible; i.e., it makes
the elements $u$ quasiinvertible in $\En(B).$
If we could say that they were quasiinvertible in $M(B),$
this would make $M(B)$ radical.
We can't initially say that; but we shall find that the
quasiinvertibility of these images in $\En(B)$
allows us to extend the domain of our map $M(A)\to\En(B)$ to a radical
subalgebra of $\En(A)$ containing $M(A).$
Since a homomorphic image of a radical algebra is radical,
we get a radical subalgebra of $\En(B)$ containing $M(B).$
Using once more the finite length assumption on $B,$ we will conclude
that that subalgebra of $\En(B)$ is nilpotent, hence so is $M(B);$
hence so is $B,$ yielding our main result (first paragraph of abstract).

It is curious that in the above development, before
assuming that $B$ had finite length, we could conclude that the maps
$1+u$ $(u\in M(B))$ were surjective, but not that they were invertible.
Let me sketch a concrete example (to be given in detail in \S\ref{S.eg})
showing how injectivity can fail, and why surjectivity must
nonetheless hold.

Suppose one takes an inverse limit $A$ of nilpotent associative
algebras $A_i,$ and divides out by the two-sided ideal $(r)$
generated by an element of the form
\begin{equation}\begin{minipage}[c]{35pc}\label{d.r=y-xyz}
$r\ =\ y\,-\,x\,y\,z\ =\ (1-l_x r_z)(y),$
\end{minipage}\end{equation}
where $x,y,z\in A.$
In the resulting algebra $A/(r),$ let us, by abuse of notation, use
the same symbols
$x,y,z$ for the images of the corresponding elements of $A.$
Thus, in that algebra we have $y=x\,y\,z;$ equivalently, $y$
is annihilated by the operator $1-l_x r_z.$
This will show that the latter operator is not injective if
we can show that $y\neq 0$ in $A/(r),$ in other words,
that $y\notin(r)$ in $A.$

Now we can formally solve~(\ref{d.r=y-xyz}) for $y,$ getting
$y=r+xrz+x^2rz^2+\dots\,;$ and in the nilpotent algebras $A_i$ of which
$A$ is the inverse limit, that equation is literally true,
since the images of $x$ and $z$ are nilpotent; so the image of $y$
in each of those algebras does lie in the image of $(r).$
But as we pass to larger and larger algebras $A_i,$
the number of terms needed in this solution can grow without bound,
so that there is no evident way to express $y\in A$ as a member
of the ideal $(r);$ and indeed,
we shall show in \S\ref{S.eg} that for appropriate choice
of these algebras and elements, it does not belong to
that ideal, so that on $A/(r),$ $1-l_x r_z$ is non-injective.
By the above considerations,
this makes $A/(r)$ non-residually-nilpotent; a quicker way
to see this is to note that the equations $y=x\,y\,x=x^2y\,x^2=\dots$
show that $y\in\bigcap_n (A/(r))^n,$ whence its image in any nilpotent
homomorphic image of $A/(r)$ must be zero.

On the other hand, I claim that whenever $x$ and $z$ are elements
of a homomorphic image $A/U$ of an inverse limit $A$ of nilpotent
associative
algebras $A_i,$ the operator $1-l_x r_z$ will be {\em surjective}.
Given an element $y\in A/U$ which we want to show is in the range
of this operator, let us lift $x,\,y,z$ to elements of $A,$
which we will denote by the same symbols.
Seeking an element $w\in A$ mapped by $1-l_x r_z$ to $y,$
we get the same sort of formal expression as before,
$w=y+xyz+x^2yz^2+\dots\,.$
Again, this sum cannot be evaluated using the
algebra operations of $A;$ but it can in each of the $A_i,$
and we find that the resulting elements of the $A_i$ yield, in
the inverse limit algebra $A,$
an element $w$ satisfying $y=w-xwx,$ as desired.

\section{Acknowledgements, and some related literature.}\label{S.ackn}

I am indebted to Nazih Nahlus for conjecturing the main result
of this note in the case where the $A_i$ are finite-dimensional
Lie algebras over an algebraically closed field of characteristic~$0.$
I am also grateful to Georgia Benkart, Christian Jensen,
Karl H. Hofmann, Greg Marks and Nazih Nahlus
for helpful comments on earlier drafts of this note.

For some related results on homomorphic images of {\em direct products}
of algebras, see~\cite{prod_Lie1},~\cite{prod_Lie2}.

In \cite{KHH+SAM}, the structures of inverse limits of
finite-dimensional Lie groups and Lie algebras are studied,
though with somewhat different emphases from this note, focusing on the
inverse limit topology, and continuous homomorphisms.

\section{Basic properties of nilpotence.}\label{S.np}

The condition of nilpotence for a nonassociative algebra $A$ can
be characterized in several ways.

In what follows, whenever $B$ and $C$ are $\!k\!$-submodules
of $A,$ we understand $BC$ to mean the $\!k\!$-submodule of $A$
spanned by all products $bc$ $(b\in B,\,c\in C).$
Let us define recursively $\!k\!$-submodules
$A_{[n]}$ and $A_{(n)}$ $(n=1,2,\dots)$ of any algebra $A$ by
\begin{equation}\begin{minipage}[c]{35pc}\label{d.weak_series}
$A_{[1]}\ =\ A,\qquad A_{[n+1]}\ =\  A\,A_{[n]}+A_{[n]}\,A,$
\end{minipage}\end{equation}
\begin{equation}\begin{minipage}[c]{35pc}\label{d.strong_series}
$A_{(1)}\,\ =\ A,\qquad A_{(n+1)}\ =
\ \sum_{0<m<n+1}\,A_{(m)}\,A_{(n+1-m)}.$
\end{minipage}\end{equation}

It is easy to see by induction that these yield
descending chains of submodules:
\begin{equation}\begin{minipage}[c]{35pc}\label{d.descend}
for $n>0,$\quad $A_{[n]}\ \supseteq\ A_{[n+1]}$\quad
and\quad $A_{(n)}\ \supseteq\ A_{(n+1)},$
\end{minipage}\end{equation}
and also that
\begin{equation}\begin{minipage}[c]{35pc}\label{d.[]<()}
for all $n,\quad A_{[n]}\ \subseteq\ A_{(n)}.$
\end{minipage}\end{equation}

If $A$ is associative, then $A_{[n]}$ and $A_{(n)}$ clearly
coincide, their common value being the submodule of $A$ spanned
by all $\!n\!$-fold products, which we shall write $A^n.$
In the next lemma, for an arbitrary algebra $A,$
we apply this notation to the associative
algebra $M(A)\subseteq\En(A),$ defined in~\S\ref{S.defs}.

\begin{lemma}\label{L.nilp}
If $A$ is an algebra, then the following conditions are equivalent:
\vspace{.2em}

\strt\ \textup{(i)}~ There exists a positive integer $n_1$ such
that $A_{[n_1]}=\{0\}.$
\vspace{.2em}

\strt\,\textup{(ii)}~ There exists a positive integer $n_2$ such
that $A_{(n_2)}=\{0\}.$
\vspace{.2em}

\textup{(iii)}~ There exists a positive integer $n_3$ such
that $M(A)^{n_3}=\{0\}.$
\vspace{.2em}

Moreover, if the above equivalent conditions hold, then
letting $N_1,\ N_2,\ N_3$ be the smallest values of $n_1,\ n_2,\ n_3$
for which the equations in those conditions are satisfied, we have
\begin{equation}\begin{minipage}[c]{35pc}\label{d.N1N2N3}
$N_3\ =\ \max(1,\,N_1-1),\qquad N_1\ \leq\ N_2\ \leq\ 2^{N_1-2}+1.$
\end{minipage}\end{equation}
\end{lemma}

\begin{proof}
We will first establish the stated relations between
conditions~(i) and~(iii), and between $N_1$ and $N_3.$
Let $l_A\subseteq M(A)$ denote the $\!k\!$-submodule of all
left-multiplication operators $l_x$ $(x\in A),$ and
$r_A$ the $\!k\!$-submodule of all right-multiplication operators $r_x.$
We claim that
\begin{equation}\begin{minipage}[c]{35pc}\label{d.M(A)^n+1}
for all $n>0,$\quad $M(A)^{n+1}\ =\ (l_A+r_A)\,M(A)^n.$
\end{minipage}\end{equation}
Here ``$\supseteq$'' is clear.
To see ``$\subseteq$'', note that $M(A)$ consists of all sums
of products of one or more elements of $l_A + r_A,$
hence $M(A)^{n+1}$ consists of all sums of products
of $n+1$ or more such elements.
If such a product has more than $n+1$ such factors, we can, in view of
associativity, group them
into $n+1$ subproducts, of which the first is a single factor.
(The assumption $n>0$ assures us that the first of $n+1$
factors is not the only one.)
So written, our product clearly belongs to $(l_A+r_A)M(A)^n,$
proving~``$\subseteq$''.

Now the recursive step
of~(\ref{d.weak_series}) says that $A_{[n+1]}=(l_A+r_A)A_{[n]},$
so using~(\ref{d.M(A)^n+1}), and induction from the case $n=1,$
one concludes that
\begin{equation}\begin{minipage}[c]{35pc}\label{d.M^nA}
for all $n>0,$\quad $A_{[n+1]}\ =\ M(A)^n(A)$
\end{minipage}\end{equation}
This gives the equivalence of~(i) and~(iii) on the
one hand, and the initial equality of~(\ref{d.N1N2N3}) on the other.

Turning to the submodules $A_{(n)},$ the inclusion~(\ref{d.[]<()})
yields the implication (ii)$\!\implies\!$(i) and the first
inequality of~(\ref{d.N1N2N3}).
To get the reverse implication and the final
inequality of~(\ref{d.N1N2N3}), we first note that these two
statements hold trivially if $A=\{0\},$ in which case $N_1=N_2=1.$
To prove them for nonzero $A,$ in which case any $n_1$ as
in~(i), or $n_2$ as in~(ii), must be $\geq 2,$ it suffices to show that
\begin{equation}\begin{minipage}[c]{35pc}\label{d.2^n-2+1}
for $n\geq 2,$\quad $A_{(2^{n-2}+1)}\ \subseteq\ A_{[n]}.$
\end{minipage}\end{equation}

For $n=2,$ we have equality.
Assuming we know~(\ref{d.2^n-2+1}) for some $n\geq 2,$ we look at
the definition of $A_{(2^{n-1}+1)}$
as in~(\ref{d.strong_series}), and note that in each of
the summands $A_{(m)}A_{((2^{n-1}+1)-m)},$ one of the indices
$m$ or $(2^{n-1}+1)-m$ will be $\geq 2^{n-2}+1,$ while the other
will be at least $1;$ hence the summand will be contained
in $A_{(2^{n-2}+1)}A+A\,A_{(2^{n-2}+1)}.$
By inductive hypothesis,
this is $\subseteq A_{[n]}A+A\,A_{[n]}=A_{[n+1]},$ as required.
\end{proof}

(If we think of an arbitrarily parenthesized nonassociative
product as representing
a binary tree of multiplications, the last part of the above proof
is essentially a calculation
showing that a binary tree with $2^{n-2}+1$ leaves $(n\geq 2)$
must contain a chain with $n$ nodes.)

\begin{definition}\label{D.nilp}
An algebra $A$ satisfying the equivalent conditions
of Lemma~\ref{L.nilp} will be called {\em nilpotent}.
\end{definition}

Lemma~\ref{L.nilp} now gives

\begin{corollary}\label{C.nilpiff}
For any algebra $A,$ $M(A)$ is nilpotent if and only if
$A$ is nilpotent.\qed
\end{corollary}

(In the sketch in the preceding section, we defined
nilpotence in terms of condition~(ii) of Lemma~\ref{L.nilp},
as is often done.
The verification that this is equivalent to~(iii) required
the ``$2^{n-2}+1$'' part of the proof of that
lemma, which is why we referred to it as not quite trivial.)
\vspace{.5em}

We end this section with some observations on nilpotence
that will not be needed for our main results.

In the inequality $N_1\leq N_2$ of~(\ref{d.N1N2N3}),
we have equality whenever $A$ is associative
by the sentence following~(\ref{d.[]<()}).
For examples where the upper bound $N_2\leq 2^{N_1-2}+1$ is
achieved, take any positive integer $n,$
and consider the (nonassociative) $\!k\!$-algebra $A$ such that
\begin{equation}\begin{minipage}[c]{35pc}\label{d.xixi}
$A$ is free as a $\!k\!$-module on a basis $\{x_1,\dots,x_{n-1}\},$
with multiplication given by $x_m\,x_m\ =\ x_{m+1}$ for
$1\leq m\leq n-2,$ and all other products of basis
elements equal to zero (including $x_{n-1}\,x_{n-1}).$
\end{minipage}\end{equation}

It is easy to verify by induction that for every $i\leq n,$
$A_{[i]}$ is the submodule spanned by $\{x_i,\dots,x_{n-1}\}.$
In particular, $A_{[i]}$ becomes $\{0\}$ starting with $i=n,$
so the $N_1$ of Lemma~\ref{L.nilp} is $n$ for this algebra.
Less obvious, but no harder to verify, is the statement that
\begin{equation}\begin{minipage}[c]{35pc}\label{d.A()=}
for every $i\leq n,$ and $j$ with $2^{i-2}<j\leq 2^{i-1},$ $A_{(j)}$
is the submodule spanned by $\{x_i,\dots,x_{n-1}\}.$
\end{minipage}\end{equation}
Indeed, note that if $i>1,$ and $j$ lies in
the above range, then $j$ can be written as the sum of two
integers $\leq 2^{i-2},$ but not as the sum of two
integers $\leq 2^{i-3}.$
Using this fact, and the definitions~(\ref{d.strong_series})
and~(\ref{d.xixi}), one gets~(\ref{d.A()=}) by induction on $i.$
So for this algebra, $N_2=2^{n-2}+1=2^{N_1-2}+1.$

The next lemma shows that Lie algebras
behave like associative algebras in this respect.

\begin{lemma}\label{L.[]=()}
\textup{(i)}~ If $A$ is an associative or Lie algebra, then
for all positive integers $p$ and $q,$
$A_{[p]}A_{[q]}\subseteq A_{[p+q]}.$

\textup{(ii)}~ If $A$ is any algebra
for which the conclusion of\textup{~(i)} holds, then
for every positive integer $n,$ $A_{[n]}=\nolinebreak A_{(n)}.$
\end{lemma}

\begin{proof}
For associative algebras,~(i) is a weakened version of the familiar
identity $A^p A^q = A^{p+q}.$

For Lie algebras, let us switch to bracket notation, and note
that by anticommutativity, the recursive step of
definition~(\ref{d.weak_series})
can be written $A_{[n+1]}=[A,\,A_{[n]}].$
This immediately gives~(i) for $p=1$ and arbitrary $q.$
So let $p>1,$ and assume inductively that the result is
true for all smaller $p.$
Using the Jacobi identity at the second step below, and
that inductive assumption at the third and fourth steps, we compute
\begin{equation}\begin{minipage}[c]{35pc}\label{d.jacobi}
$
[A_{[p]},\,A_{[q]}]\ =
\ [\,[A,\,A_{[p-1]}],\,A_{[q]}]\ \subseteq
\ [A,\,[A_{[p-1]},\,A_{[q]}]\,]\ +
\ [A_{[p-1]},\,[A,\,A_{[q]}]\,]\\[.3em]
\strt\qquad\subseteq\ [A,\,A_{[p+q-1]}]\ +
\ [\,A_{[p-1]},\,A_{[q+1]}]\ \subseteq
\ A_{[p+q]}\ + \ A_{[p+q]}\ =
\ A_{[p+q]}.$
\end{minipage}\end{equation}

To get~(ii), recall from~(\ref{d.[]<()}) that
$A_{[n]}\subseteq A_{(n)}$ for arbitrary algebras, and note that
by definition we have equality when $n=1.$
Thus, it suffices to prove the inclusion
$A_{[n]}\supseteq A_{(n)}$ when $n>1,$
inductively assuming this inclusion for smaller $n.$
The inclusion we are to
prove is clearly equivalent to the statement that each
summand $A_{(m)}A_{(n-m)}$ in the definition of $A_{(n)}$
is contained in $A_{[n]}.$
By our inductive
hypothesis, $A_{(m)}A_{(n-m)}$ is contained in $A_{[m]}A_{[n-m]},$
and by the assumed condition from~(i), this is
indeed contained in~$A_{[n]}.$
\end{proof}

\section{Properties of $M(A).$}\label{S.MA}

As noted in~\S\ref{S.defs},
one defines the {\em multiplication algebra} $M(A)$ of any
algebra $A$ to be the (generally nonunital) subalgebra
of the associative algebra $\En(A)$ generated by the
left and right multiplication operators $l_x$ and $r_x,$
as $x$ ranges over $A.$

For a general homomorphism of algebras $h:A\to B,$
there is no natural way to map $M(A)$ to $M(B).$
For instance, if $h$ is the inclusion of a subalgebra $A$ in
an algebra $B,$ such that a central element $x\in A$
becomes noncentral in $B,$ then $l_x=r_x$ in $M(A),$ but the
corresponding members of $M(B)$ are distinct.
For surjective homomorphisms, however, this problem goes away:

\begin{lemma}\label{L.surj}
If $h:A\to B$ is a surjective homomorphism of algebras, then
there exists a unique homomorphism $M(h):M(A)\to M(B)$ such that
\begin{equation}\begin{minipage}[c]{35pc}\label{d.Mfx}
for all $x\in A,\quad M(h)(l_x)\ =\ l_{h(x)}$\quad
and \quad $M(h)(r_x)\ =\ r_{h(x)},$
\end{minipage}\end{equation}
equivalently, such that
\begin{equation}\begin{minipage}[c]{35pc}\label{d.Mfu}
for all $u\in M(A)$ and $a\in A,$\quad $(M(h)(u))(h(a))\ =\ h(u(a)).$
\end{minipage}\end{equation}

Moreover, $M(h)$ is surjective.
\end{lemma}

\begin{proof}
$\mathrm{Ker}(h)$ is an ideal of $A,$ hence it is
carried into itself by every map $l_x$ and every map $r_x,$ and thus
by every element $u$ of the algebra $M(A)$ generated by such maps.
Hence if two elements $a,\ a'\in A$ differ by an element
of $\mathrm{Ker}(h),$ so do $u(a)$ and $u(a');$ that is,
if $h(a)=h(a'),$ then $h(u(a))=h(u(a'));$ so as $B=h(A),$ we get
a well-defined linear map $M(h)(u):B\to B$ satisfying~(\ref{d.Mfu}).

It is immediate that $M(h)$ is an algebra homomorphism, and
acts by~(\ref{d.Mfx}) on elements $l_x$ and $r_x.$
It is surjective because it carries the generating
set $\{l_x,\,r_x\mid x\in A\}$ of $M(A)$ to the corresponding
generating set of $M(B).$
\end{proof}

It is also immediate that for a composable pair of surjective
algebra homomorphisms $h,\ g,$ we have $M(hg)=M(h)M(g);$
and that if we write $\mathrm{id}_A$ for
the identity homomorphism $A\to A,$ then
$M(\mathrm{id}_A)=\mathrm{id}_{M(A)}.$
Thus, $M$ is a functor from the category
whose objects are $\!k\!$-algebras and whose morphisms
are {\em surjective} algebra homomorphisms to the category of
associative $\!k\!$-algebras.

Now suppose we are given an inverse system of $\!k\!$-algebras; i.e.,
that for some inversely directed partially ordered set $I,$ we are
given a family of algebras $(A_i)_{i\in I}$ and algebra
homomorphisms $f_{ji}: A_i\to A_j$ $(i\leq j),$ such that
$f_{ii}=\mathrm{id}_{A_i}$ for $i\in I,$ and
$f_{kj}f_{ji}=f_{ki}$ for $i\leq j\leq k.$
Recall that the {\em inverse limit} of this system can be
constructed (or alternatively, the reader may consider it to be defined)
as the subalgebra $A=\lm_I A_i\subseteq \prod_I A_i$ consisting
of those elements $(a_i)_{i\in I}$ such that
$f_{ji}(a_i)=a_j$ for all $i\leq j.$
Thus, the projection maps $p_j: A\to A_j$ carrying
$(a_i)_{i\in I}$ to $a_j\in A_j$ satisfy
\begin{equation}\begin{minipage}[c]{35pc}\label{d.fjipi}
$f_{ji}\,p_i\ =\ p_j$ $(i\leq j).$
\end{minipage}\end{equation}
The algebra $A,$ with these maps, is universal for~(\ref{d.fjipi})
(see~\cite[\S\S7.4-7.5]{245} for motivation and details).

For a general inverse system of algebras $A_i,$ we cannot talk
of applying $M$ to the $f_{ji}$ and $p_i,$
since these may not be surjections.
(Even if all the $f_{ji}$ are surjective, the resulting
$p_i$ may fail to be \cite{LH} \cite{GH+AHS}~\cite{WW}.)
However, given any inverse system of
algebras $(A_i)_{i\in I},$ and writing $A$ for
its inverse limit, if we
replace each $A_i$ with its subalgebra $p_i(A),$ the result
will be an inverse system still having inverse limit $A,$ but where
the restricted maps $f_{ji}$ and $p_i$ are all surjective.
(Actually, surjectivity of the $p_i$ implies
surjectivity of the $f_{ji},$ in view of~(\ref{d.fjipi}).)
Also, of course, if the original algebras $A_i$ were nilpotent,
the subalgebras with which we have replaced them will still be.
Hence in what follows, we shall often restrict attention
to inverse systems of algebras in which all these maps are surjective.

\begin{lemma}\label{L.M+lm}
Let $(A_i, f_{ji})_{i,j\in I}$ be an inverse system of $\!k\!$-algebras,
and $A=\lm A_i$ its inverse limit, with projection maps $p_i:A\to A_i;$
and suppose the $p_i$ \textup{(}and hence the $f_{ji})$ are all
surjective.

Then $\lm_I M(A_i)$ may be identified with a subalgebra of $\En(A)$
containing $M(A),$ by letting each $(u_i)_{i\in I}\in\lm_I M(A_i)$
act on $A$ by sending $(a_i)_{i\in I}\in A$ to
$(u_i(a_i))_{i\in I}\in A.$
\end{lemma}

\begin{proof}
The condition for $(a_i)_{i\in I}$ to belong to $A=\lm_I A_i$ says that
each $f_{ji}$ takes $a_i$ to $a_j,$ and the condition for
$(u_i)_{i\in I}$ to belong to $\lm_I M(A_i)$ says that
each $M(f_{ji})$ takes $u_i$ to $u_j.$
By~(\ref{d.Mfu}), with $f_{ji}$ for $h,$ the latter condition
tells us that $u_j(f_{ji}(a_i))=f_{ji}(u_i(a_i)),$ and by the former,
the left-hand side of this relation equals $u_j(a_j).$
This shows that the $\!I\!$-tuple $(u_i(a_i))_{i\in I}$
again belongs to $A=\lm_I A_i.$
Thus, each $u\in\lm_I M(A_i)$ induces a map $A\to A$
acting as described in the last phrase of the lemma.

It is routine to verify that these maps are module endomorphisms,
that this action of $\lm_I M(A_i)$ on $A$ respects the ring operations
of $\lm_I M(A_i),$ and that it is faithful; so we get an
identification of $\lm_I M(A_i)$ with a subalgebra of $\En(A).$
Finally, for any $x=(x_i)_{i\in I}\in A,$ one verifies
that $(l_{x_i})_{i\in I}$
is an element of $\lm_I M(A_i)$ that acts on $A$ as $l_x;$
so as a subalgebra of $\En(A),$ $\lm_I M(A_i)$ contains each
operator $l_x.$
It similarly contains each $r_x,$ hence it contains $M(A).$

In fact, one easily verifies that each $u\in M(A)$ agrees
with the element $(M(p_i)(u))_{i\in I}\in\lm_I M(A_i).$
\end{proof}

In general, $\lm_I M(A_i)$ will be properly larger than $M(A).$
Indeed, as noted in~\S\ref{S.preview}, if all $A_i$ are nilpotent,
then the algebras $M(A_i)$ are nilpotent, hence are radical, hence
$\lm_I M(A_i)$ is radical.
But in the example we sketched there (to be given in
detail in~\S\ref{S.eg}), $M(A)$ was
not radical (since the image of $1-\,l_x\,r_z$ under the
map $M(A)\to M(B)$ was not invertible, so
that element could not have been invertible in $M(A)).$
Thus, in such an example, $M(A)$ cannot coincide with $\lm_I M(A_i),$
and, indeed, must fail to be closed therein under quasiinverses.

\section{Hopfian modules, and modules of finite length.}\label{S.Hopf+}

As also noted in~\S\ref{S.preview}, the operator $1-l_x\,r_z$
of the example referred to above will nevertheless
be {\em surjective} on any homomorphic image $B$ of $A.$
A key to the proof of our main result will be to restrict
attention to algebras $B$ whose $\!k\!$-module structure
is such that every surjective module endomorphism is invertible.
In getting our main conclusion, we will have to make the
stronger assumption that $B$ has finite length as a $\!k\!$-module;
but let us take a look at the weaker condition just stated,
under which we will be able to carry our proof part of the way.

An algebraic structure is said to be {\em Hopfian} if it has no
surjective endomorphisms other than automorphisms \cite{VAH} \cite{KV}.
Here are some quick examples of Hopfian modules:
A vector space is Hopfian if and only if it is finite-dimensional.
A Noetherian module $M$ over any ring is Hopfian;
for if $h:M\to M$ were surjective but not injective, then the chain
\begin{equation}\begin{minipage}[c]{35pc}\label{d.strict_chain}
$\{0\}\ \subsetneq\ h^{-1}(\{0\})\ \subsetneq
\ h^{-1}(h^{-1}(\{0\}))\ \subsetneq\ \dots$
\end{minipage}\end{equation}
would contradict the Noetherian condition \cite[Prop.~IV.5.3(i)]{HB}
\cite[Prop.~6(i)]{VAH} \cite[Prop.~1.14]{TYL}.
In particular, any module of finite length is Hopfian.
Over a commutative ring, every finitely generated module
is Hopfian \cite[Prop.~IV.5.3(ii)]{HB}, and over
a commutative integral domain $k$ with field of fractions $F,$
any $\!k\!$-submodule of a finite-dimensional
$\!F\!$-vector-space is Hopfian (cf.~\cite[Prop.~11]{VAH}).
So, for instance, $\mathbb{Q}$ is a Hopfian $\!\mathbb{Z}\!$-module --
though its homomorphic image $\mathbb{Q/Z}$ is an example
of a non-Hopfian module.
(The classes of Hopfian modules listed above are all closed under
finite direct sums; however, examples are known of non-Hopfian finite
direct sums of Hopfian modules; indeed, of a Hopfian abelian group
$A$ such that $A\oplus A$ is not Hopfian~\cite{ALSC}.)

The next result only considers module-structures on $A$ and $B,$
and does not require the base-ring to be commutative.
In view of our convention that $k$ denotes a commutative
ring, we shall call the base-ring $K.$
(In our application of the result, however, $K$
will be our commutative ring $k.)$

\begin{proposition}\label{P.Hopfian}
Suppose $A$ and $B$ are right modules over an associative
unital ring $K,$ let $h:A\to B$ be a surjective module
homomorphism, and let $\En(A;\ker(h))$ be the
subring of the endomorphism ring $\En(A)$ consisting
of the endomorphisms that carry $\ker(h)$ into itself
\textup{(}and hence induce endomorphisms of~$B).$

Suppose $R$ is a radical subring
of $\En(A),$ and $B$ is Hopfian as a $\!K\!$-module.
Then $R\cap\nolinebreak\En(A;\ker(h))$ is also a radical ring; hence its
image in $\En(B)$ is a radical subring of $\En(B).$
\end{proposition}

\begin{proof}
To show that the ring $R\cap\En(A;\ker(h))$ is radical, it suffices
to verify that it is closed under quasi\-inverses in $R.$
Let $r\in R\cap\En(A;\ker(h)),$ and $s$ be its quasiinverse in $R.$
Thus, $1+r$ and $1+s$ are mutually inverse elements of $\En(A).$

Since $1+r$ is invertible as an endomorphism of $A,$ it is
in particular surjective, from which it is easy to see that the
endomorphism of $B$ it induces is surjective.
Since $B$ is Hopfian, that endomorphism is also injective, and
this says that back in $\En(A),$
$1+r$ carries no element from outside $\ker(h)$ into $\ker(h).$
Thus, the inverse map $1+s\in\En(A)$ carries no element of $\ker(h)$
out of $\ker(h),$ i.e., $1+s\in\En(A;\ker(h));$
hence $s\in\En(A;\ker(h)),$ as required.
\end{proof}

We shall use the above result in conjunction with
part~(iii) of the next lemma.
Note that in that lemma, we return to the general hypothesis
of a commutative base-field $k;$ and that
parts~(i) and~(ii), but not part~(iii), assume $B$ an algebra.
(Even in part~(iii), it will be an algebra in our application.)

\begin{lemma}\label{L.idpt}
\textup{(i)}~ In a radical associative algebra $B,$ a finite set of
elements $X\subseteq B$
which are not all zero cannot satisfy $X\subseteq BX.$

\textup{(ii)}~ A radical associative algebra $B$ cannot contain
a nonzero finitely generated idempotent subalgebra $S=\nolinebreak S^2.$

\textup{(iii)}~ If $B$ is a $\!k\!$-module of finite length, then
any radical subalgebra $R\subseteq\En(B)$ is nilpotent.
\end{lemma}

\begin{proof}
(i): Writing $U=(k+B)X$ for the left ideal of $B$ generated by $X,$
the condition $X\subseteq BX$ implies that $BU=U;$
so by Nakayama's Lemma \cite[Lemma~4.22(2)]{TYL1}
\cite[Exercise XVII.7.4, p.661]{SL.Alg},
$U=\{0\},$ hence $X\subseteq\{0\},$ contradicting our hypothesis.
(The references cited state Nakayama's Lemma for unital rings.
In our present context, we can apply that version of the
lemma to the left module $U$ over the unital ring
$k+B,$ in which $B$ is an ideal contained in the radical.)

(ii): Suppose $S$ were an idempotent subalgebra of
$B$ generated as a $\!k\!$-algebra by a finite set $X.$
The fact that $S$ is generated by $X$ implies that
$S\subseteq(k+S)X;$ hence
\begin{equation}\begin{minipage}[c]{35pc}\label{d.X<...}
$X\ \subseteq\ S\ =\ S^2\ \subseteq B\,S\ \subseteq
\ B(k+S)X\ \subseteq\ B\,X,$
\end{minipage}\end{equation}
contradicting~(i).

(iii): Since $B$ has finite length as a $\!k\!$-module, the chain of
submodules $B\supseteq RB\supseteq R^2B\supseteq\dots$ stabilizes;
say $R^{n+1}B=R^nB.$
Again using finite length of $B$ we see that $R^nB$ is finitely
generated as a $\!k\!$-module, hence as an $\!R\!$-module;
hence, since it is carried onto itself
by the radical ring $R,$ Nakayama's lemma shows that it is zero.
Hence $R^n=\{0\}.$
\end{proof}

\section{The main theorem.}\label{S.main}

\begin{definition}\label{D.pro-np}
A $\!k\!$-algebra $A$ which can be written as
an inverse limit of nilpotent $\!k\!$-algebras will be called
{\em pro-nilpotent}.
\end{definition}

Part~(iii) of the next result is what we have been aiming at.
The first two parts note what can be said under weaker assumptions.

\begin{theorem}\label{T.main}
Let $B=h(A)$ be a surjective homomorphic image of a pro-nilpotent
$\!k\!$-algebra $A.$
Then
\vspace{.2em}

\textup{(i)}~ For every $r\in M(B),$ the operator $1+r\in\En(B)$
is surjective.
\textup{(}More generally, for every $n>0$ and
$r\in\mathrm{Mat}_n(M(B)),$ $1+r$ acts surjectively on the
direct sum of $n$ copies of $B.)$
\vspace{.2em}

\textup{(ii)}~ If $B$ is Hopfian as a $\!k\!$-module, $M(B)$ is
contained in a Jacobson radical subalgebra of $\En(B).$
\vspace{.2em}

\textup{(iii)}~ If $B$ is of finite length as a $\!k\!$-module,
then it is nilpotent as an algebra.
\end{theorem}

\begin{proof}
Let $A=\lm_I A_i,$ where $(A_i, f_{ji})_{i,j\in I}$ is an inverse
system of nilpotent $\!k\!$-algebras.

As noted earlier, if we replace each $A_i$ by the image $p_i(A)$
therein, and restrict the $f_{ji}$ to these subalgebras,
we get a new inverse system having the same inverse limit $A,$
and such that the restricted maps $p_i$ and $f_{ji}$ are surjective;
moreover, the new $A_i,$ being subalgebras of the given algebras,
are still nilpotent.
Hence without loss of generality, let us assume all
the $p_i$ and $f_{ji}$ surjective.

By Corollary~\ref{C.nilpiff}, the multiplication algebras $M(A_i)$
are nilpotent, hence are radical, and an inverse
limit of radical rings is radical; so under the identification
of Lemma~\ref{L.M+lm}, $\lm_I M(A_i)$ is a radical subalgebra
of $\En(A)$ containing $M(A)\subseteq\En(A;\ker(h)).$

With no additional assumptions, we see that the radicality
of $\lm_I M(A_i)$ implies that for every $u\in M(A),$ the
operator $1+u$ is invertible on $A,$ hence in particular,
acts surjectively, hence that its image in $M(B)$ acts surjectively
on $B,$ giving the first statement of~(i).
This argument applies, more generally, to $\mathrm{Mat}_n(M(A))$ and
$\mathrm{Mat}_n(M(B)),$ acting on a direct sum of copies of
$A,$ respectively $B,$ yielding the parenthetical generalization.

If $B$ is Hopfian as a $\!k\!$-module, then
by Proposition~\ref{P.Hopfian},
$(\lm_I M(A_i))\cap\En(A;\ker(h))$ is a radical $\!k\!$-algebra.
Since $M(A)\subseteq (\lm_I M(A_i))\cap\En(A;\ker(h)),$
its image $M(B)=M(h)(M(A))\subseteq\En(B)$ is contained in the radical
subalgebra $M(h)((\lm_I M(A_i))\cap\En(A;\ker(h))),$ giving~(ii).

Finally, if $B$ has finite length, Lemma~\ref{L.idpt}(iii) shows
that the above radical subalgebra of $\En(B)$ is nilpotent, hence its
subalgebra $M(B)$ is nilpotent,
hence by Corollary~\ref{C.nilpiff} again, $B$ is nilpotent.
\end{proof}

In the next section we will give counterexamples to the
conclusions of Theorem~\ref{T.main} and Lemma~\ref{L.idpt} in
the absence of some of the hypotheses; in particular, the
finite length hypothesis of Theorem~\ref{T.main}(iii).
On the other hand, in~\S\ref{S.Nak-like} (after some general
observations in~\S\ref{S.chain}), we will get a few
additional positive results.
In~\S\ref{S.solv} we note a consequence of our main theorem
for solvable Lie algebras,
and in the final~\S\S\ref{S.variants}-\ref{S.questions},
some questions and topics for further study.

\section{Counterexamples.}\label{S.eg}

The first example below will be the promised case of a
homomorphic image $B$ of a pro-nilpotent algebra containing elements
$x,\,z$ such that the map $1-l_x r_z\in M(B)$ is not one-to-one.

In constructing that example and the next,
we shall make use of unital free associative
algebras $k\<\langle X\rangle$ in a finite set $X$ of
noncommuting indeterminates (e.g., $X=\{x,y,z\})$ over a field $k,$
their completions, which are noncommuting formal power series algebras
$k\<\langle\langle X\rangle\rangle,$
and the nonunital versions of these two
constructions (their ``augmentation ideals'', i.e., the kernels of the
unital homomorphisms to $k$ sending the indeterminates to zero),
which we will denote $[k\<]\langle X\rangle,$ respectively
$[k\<]\langle\langle X\rangle\rangle.$

Within these algebras, we shall write
$(a,b,\dots)$ for the $\!2\!$-sided ideal generated by
elements $a,b,\dots\,.$
In the completed algebras, we shall also
write $((a,b,\dots))$ for the {\em closure} of such an ideal in the
inverse limit topology.

These examples will start by taking
a set $T$ of monomials in the given free generators,
which does not contain the monomial $1,$ and
forming the factor algebra $k\<\langle X\rangle/(T).$
Note that this has a $\!k\!$-basis consisting of all monomials not
containing any subword belonging to $T.$
We shall then form the
completion $k\<\langle\langle X\rangle\rangle/((T))$
and take for our $A$ the subalgebra
$[k\<]\langle\langle X\rangle\rangle/((T)).$
It is not hard to see that $k\<\langle\langle X\rangle\rangle/((T))$
is the inverse limit of the factor-algebras
$k\<\langle X\rangle/(T\cup X^i)$
where $X^i$ denotes the set of monomials of length $i$ in the
given generators, so that
$[k\<]\langle\langle X\rangle\rangle/((T))$ is the inverse limit of
the nilpotent algebras $A_i = [k\<]\langle X\rangle/(T\cup X^i).$
This inverse limit consists of all formal infinite
$\!k\!$-linear combinations of monomials having no subword in $T.$

By abuse of notation, we shall use the same symbols
$x,\,\dots$ for our original generators and
for their images in our various factor-algebras.

\begin{example}\label{E.y-xyz}
There exists a pro-nilpotent associative algebra $A$ over a field $k$
having elements $x,y,z$ such that $y\notin (y-xyz).$

Thus, in the algebra $B=A/(y-xyz),$
the operator $1-l_x r_z$ annihilates the nonzero element $y.$
In particular, $0\neq y\in B\,y\,B,$ so $B$
cannot be residually nilpotent.

Hence also, though the algebras $A$ and $B$ are Jacobson radical,
$M(A)$ and $M(B)$ are not: in each,
the element $-l_x r_z$ is not quasiinvertible \textup{(}though
it is the product of the quasiinvertible
elements $-l_x$ and $r_z).$
\end{example}

\begin{proof}[Construction and proof]
Since it is easier to study the ideal
of an algebra $[k\<]\langle\langle X\rangle\rangle/((T))$
generated by one of the indeterminates than the ideal
generated by a more complicated expression, we shall take for $A$ an
algebra of the form $[k\<]\langle\langle x,w,z\rangle\rangle/((T)),$
find a $y\in A$ such that $w=y-xyz,$
and then obtain $B$ by dividing $A$
by the ideal generated by the indeterminate $w.$

Let the set of monomials $T$ be  chosen so that the only nonzero
monomials in $[k\<]\langle x,w,z\rangle/(T)$ are the words
\begin{equation}\begin{minipage}[c]{35pc}\label{d.xiwzj}
$x^iw\,z^j$ $(i,j\geq 0),$ and subwords of such words.
\end{minipage}\end{equation}
Thus, we take
\begin{equation}\begin{minipage}[c]{35pc}\label{d.T=1}
$T\ =\ \{xz,\,wx,\,ww,\,zw,\,zx\}.$
\end{minipage}\end{equation}
and let
\begin{equation}\begin{minipage}[c]{35pc}\label{d.A=1}
$A\ =\ [k\<]\langle\langle x,w,z\rangle\rangle/((T)).$
\end{minipage}\end{equation}
For convenient calculation with ideals, we also introduce the notation
\begin{equation}\begin{minipage}[c]{35pc}\label{d.k+A=1}
$k+A\ =\ k\<\langle\langle x,w,z\rangle\rangle/((T)).$
\end{minipage}\end{equation}

On $A,$ which by our preceding
discussion is pro-nilpotent, consider the operator
$-\,l_x r_z\in M(A)\subseteq\lm M(A_i).$
Since the latter algebra is Jacobson radical, $-\,l_x r_z$ is
quasiinvertible in $\En(A);$ so let $y=(1-l_x r_z)^{-1}(w).$
Clearly, this has the form
\begin{equation}\begin{minipage}[c]{35pc}\label{d.y=1}
$y\ =\ w+xwz+x^2wz^2+\dots+x^nwz^n+\dots\ .$
\end{minipage}\end{equation}
(Indeed, one can see immediately that this satisfies $w=y-xyz.)$

We claim that $y\notin (w).$
To see this, note that every element of $(w)$ is a {\em finite} sum
\begin{equation}\begin{minipage}[c]{35pc}\label{d.sumaiwbi}
$\sum_{i=1}^n\,a_i\,w\,b_i\quad(a_1,\dots,a_n,\,b_1,\dots,b_n\in k+A).$
\end{minipage}\end{equation}
Now an element $a$ such that every monomial occurring in $a$ contains
a factor $w$ or $z$ will annihilate $w$ on the
left, and elements in which all monomials occurring contain
factors $w$ or $x$ likewise annihilate $w$
on the right (see~(\ref{d.xiwzj}),~(\ref{d.T=1})); so let us write
each $a_i$ in~(\ref{d.sumaiwbi}) as $a_i'+a_i'',$
where $a_i'\in k[[x]],$ while the monomials occurring in
$a_i''$ all have factors $w$ or $z,$ and each
$b_i$ as $b_i'+b_i'',$ where $b_i'\in k[[z]]$ while the monomials
occurring in $b_i''$ all have factors $w$ or $x.$
Then~(\ref{d.sumaiwbi}) becomes
\begin{equation}\begin{minipage}[c]{35pc}\label{d.sumai'wbi'}
$\sum_{i=1}^n\,a_i'\,w\,b_i'\quad(a_1',\dots,a_n'\in k[[x]],
\ b_1',\dots,b_n'\in k[[z]]).$
\end{minipage}\end{equation}

We now see that if in~(\ref{d.sumai'wbi'}) we take the right
coefficient, in $k[[z]],$ of $x^j w$ for any $j\geq 0,$ this will be
a $\!k\!$-linear combination of $b_1',\dots,b_n'.$
In particular,
\begin{equation}\begin{minipage}[c]{35pc}\label{d.f.dim}
the $\!k\!$-vector-subspace of $k[[z]]$ spanned by the
right coefficients in that algebra of the words $x^j w$
$(j=0,1,\dots)$ is finite-dimensional over $k.$
\end{minipage}\end{equation}

However, by~(\ref{d.y=1}),
the right coefficient of $x^j w$ in $y$ is $z^j.$
The elements $z^j$ span an infinite-dimensional subspace of
$k[[z]],$ so $y\notin (w)=(y-xyz),$ proving our first
assertion about this example.

Since $y=xyz$ in $B,$ we have $y\in ByB\subseteq B(ByB)B\subseteq\dots,$
hence $y$ maps to $0$ in every nilpotent homomorphic image
of $B,$ so $B$ is not residually nilpotent.

Of the final assertions, $A$ is radical because
it is an inverse limit of radical algebras, while $B$ is because
it is a homomorphic image of $A.$
We have shown that $l_x r_z\in M(B)$ is not
quasiinvertible, hence the same is necessarily true of the element of
$M(A)$ denoted by the same symbol, which maps to it;
hence neither $M(A)$ nor $M(B)$ is radical.
Finally, the maps $A\to M(A)$ given by $a\mapsto l_a$
and $a\mapsto r_a$ are a homomorphism and an
antihomomorphism, so the quasiinvertibility
of $x$ and $z$ in $A$ implies quasiinvertibility of $l_x$ and $r_z$
in $M(A),$ hence also in $M(B).$
\end{proof}

We remark that if we write $A'=[k\<]\langle\langle x,w,z\rangle\rangle,$
so that $A=A'/((T)),$ and define $y=(1-l_x l_z)^{-1}(w)$
in the pro-nilpotent algebra $A',$ then the fact that
$y\notin(w)$ in $A$ implies that the same holds in $A';$
so the conclusions proved above for $A$ and $B$
also hold for $A'$ and $B'=A'/(w).$
Dividing out by $((T))$ just made it easier for us
to see what we were doing.

In a different direction, suppose that instead of dividing
a pro-nilpotent algebra $A$ by the ideal generated
by an element of the
form $r=y-xyz=(1-l_x r_z)(y),$ we had divided such an algebra
by the ideal generated by an element of the
form $s=y-xy=(1-l_x)(y).$
Using the fact that $-x,$ and hence $l_{-x},$ are quasiinvertible,
we find that in this case, $y$ {\em does} belong to the ideal $(s).$
So it goes to zero in our factor-ring; thus, this simpler
construction does not give an example of non-injectivity.
The same, of course, happens if we divide out by an element
of the form $t=y-yz=(1-r_z)(y).$
So the two-sided nature of the operator $l_x r_z$ was needed
to make Example~\ref{E.y-xyz} work.

However, the fact that $a\mapsto l_a$ is a homomorphism $A\to M(A),$
which we used in the preceding paragraph to conclude that
$l_{-x}$ was quasiinvertible, holds only for associative algebras $A.$
In the next example, we shall find that on an inverse
limit of nilpotent Lie algebras, an operator of the
form $1-l_x$ can fail to be quasiinvertible.
(In fact, that example will be ``one-sided'' from the Lie point of
view, but ``two-sided'' from the associative point of view.)

The construction will be formally a little simpler than the preceding
example, but the verification will be a bit more complicated.

\begin{example}\label{E.y-xy+yx}
There exists a pro-nilpotent associative algebra $A$ over a field $k$
having elements $x,y$ such that $y\notin (y-xy+yx).$
Thus, under commutator brackets, $A$ is a pro-nilpotent Lie algebra
with elements $x,y$ such that $y\notin (y-[x,y])_\mathrm{Lie}$
\textup{(}where $(\ )_{\mathrm{Lie}}$ means
``Lie ideal generated by''\textup{)}.

Hence, in the Lie algebra $B=A/(y-[x,y])_\mathrm{Lie},$
$1-\mathrm{ad}_x$ annihilates the nonzero element~$y.$
In particular,
$0\neq y\in [B,\,y],$ so $B$ is not residually nilpotent.

As in the previous example, the associative algebras $A$ and
$A/(y-[x,y])$ are Jacobson radical, but their multiplier algebras
are not: $r_x-l_x$ is not quasiinvertible
\textup{(}though $r_x$ and $l_x$ are\textup{)}.
\end{example}

\begin{proof}[Construction and proof]
This time, let us start with the associative
algebra $[k\<]\langle\langle x,w\rangle\rangle/((T)),$
where $T$ is chosen so that the only nonzero monomials are the words
\begin{equation}\begin{minipage}[c]{35pc}\label{d.xiwxj}
$x^iw\,x^j$ $(i,j\geq 0)$ and their subwords.
\end{minipage}\end{equation}
Thus, we take
\begin{equation}\begin{minipage}[c]{35pc}\label{d.T=2}
$T\ =\ \{w x^i w\mid i\geq 0\},$
\end{minipage}\end{equation}
and let
\begin{equation}\begin{minipage}[c]{35pc}\label{d.A=2}
$A\ =\ [k\<]\langle\langle x,w\rangle\rangle/((T)).$
\end{minipage}\end{equation}

We now define
\begin{equation}\begin{minipage}[c]{35pc}\label{d.y=2}
$y\ =\ (1-l_x+r_x)^{-1}(w)\ \in\ A.$
\end{minipage}\end{equation}

Though the obvious way to begin the calculation of this
element would be to write
$(1-l_x+r_x)^{-1}=\sum_{i=0}^\infty\ (l_x-r_x)^i,$ we can
get the coefficient of $x^i w$ in~(\ref{d.y=2}) more
quickly if we instead use the formula
\begin{equation}\begin{minipage}[c]{35pc}\label{d.1-ad-1}
$(1-l_x+r_x)^{-1}\ =\ \sum_{i=0}^\infty\ l_x^i\,(1+r_x)^{-1-i},$
\end{minipage}\end{equation}
which is valid because $l_x$ and $1+r_x$ commute.
This gives
\begin{equation}\begin{minipage}[c]{35pc}\label{d.y=3}
$y\ =\ \sum_{i=0}^\infty\ x^i\,w\,(1+x)^{-1-i}$
\end{minipage}\end{equation}
in the formal power series algebra
$A=[k\<]\langle\langle x,w\rangle\rangle/((T)).$

Again, if this lay in $(w),$ it would follow that the right
factors $(1+x)^{-1-i}$ $(i=0,1,\dots)$ of the monomials $x^i\,w$
would lie in a finite-dimensional $\!k\!$-subspace of $[k\<][[x]].$
But they not: the positive and negative
powers of $1+x$ are $\!k\!$-linearly independent in the
field $k(x),$ so they are $\!k\!$-linearly independent in the
larger formal Laurent series field
$k((x)),$ hence in the smaller
formal power series algebra $k[[x]]\subseteq A.$

Hence $y\notin (w)=(1-l_x+r_x)(y)=(y-xy+yx),$ and since the
Lie ideal generated by $y-xy+yx=y-[x,y]$ is contained in the
associative ideal generated by that element, we likewise
have $y\notin (y-[x,y])_\mathrm{Lie}.$

Again, the other assertions follow.
\end{proof}

The above example may seem suspicious:
The elements $x$ and $y$ span a $\!2\!$-dimensional sub-Lie-algebra
$B^*$ of $B,$ so
suppose we let $A^*$ be the inverse image of this algebra in $A,$
and replace each member of the family of algebras $A_i$
of which $A$ is the inverse limit by the image $A_i^*$ of $A^*$ therein.
Won't the resulting inverse system have $A^*$ as inverse limit,
giving a description of the finite-dimensional non-nilpotent
Lie algebra $B^*$ as a homomorphic image of an inverse limit $A^*$ of
nilpotent Lie algebras?

What's wrong here is the assumption that the
inverse limit of the $A_i^*$ will be $A^*.$
Rather, one finds that that inverse limit will be the closure
of $A^*$ in the inverse limit topology on $A.$
Since the map $A\to B$ is not continuous in that topology
(its kernel is $(y-[x,y])_\mathrm{Lie},$ not
$((y-[x,y]))_\mathrm{Lie}),$ the closure of $A^*$ may
have a much larger image than $B^*.$

Another thought:
Looking intuitively at Examples~\ref{E.y-xyz} and~\ref{E.y-xy+yx},
we can say that in each, we took a pro-nilpotent algebra $A,$
and were able to arrange for an element
$y\in A$ to ``survive'' under a homomorphism $A\to B$
that made it fall together with a member of $AyA$ or $Ay+yA.$
In these cases, $y$ survived ``with the help of'' other
elements $(x,$ and possibly $z)$ which did not themselves fall together
with higher-degree expressions.
We may ask whether a family $X$ of elements can
all ``help one another'' to survive under a homomorphism that makes
each fall together with a linear combination of higher degree monomials
in it and the others.
One way of posing this question is:  Can a homomorphic image $B$
of a pro-nilpotent algebra $A$ contain a nonzero
subalgebra $S$ that is idempotent, i.e., satisfies $S=S^2$?

If our algebras are associative, and the set $X$ generating
$S$ is finite, the answer is no.
Indeed, $A,$ and hence $B,$ will be
Jacobson radical, and Lemma~\ref{L.idpt}(ii) says that such an algebra
cannot have a nonzero finitely generated idempotent subalgebra.
Lemma~\ref{L.idpt}(i) describes a more general restriction.

However, these restrictions fail for nonassociative algebras.
Indeed, in Example~\ref{E.y-xy+yx} we had $y\in [B,y],$
contradicting the analog of Lemma~\ref{L.idpt}(i).
We record next a much simpler (though non-Lie) example with the same
property (which
we will want to call on for another property later), then
an example with a finite-dimensional simple subalgebra.

\begin{example}\label{E.y-xy}
Another pro-nilpotent algebra $A$ over a field $k$
having elements $x,y$ such that $y\notin (y-xy),$
hence such that on $B=A/(y-xy),$
$1-l_x$ annihilates the nonzero element $y;$ hence such
that $0\neq y\in B\,y$
\textup{(}in contrast to Lemma~\ref{L.idpt}\textup{(i))}.
\end{example}

\begin{proof}[Construction and proof]
A natural approach, paralleling our earlier constructions, would
be to start with a nonassociative $\!k\!$-algebra on
generators $x,\,w,$ in which all monomials are set to zero
except for $x,$ $w,$ $xw,$ $x(xw),$ $x(x(xw)),\dots\,.$
But rather than dealing with a free nonassociative algebra, and the
resulting proliferation of parentheses, let us simply name the
resulting basis
of our algebra, and say how the multiplication acts.
(The main value of the ``free associative algebra modulo monomials''
approach of our previous examples
was to insure that the algebra described was
associative; but no such condition is needed here.)

So let us start with an algebra having a basis
$\{x,\,w_0,\,w_1,\,w_2,\dots\},$ and multiplication given by
\begin{equation}\begin{minipage}[c]{35pc}\label{d.xwi...}
$x\,w_i\ =\ w_{i+1}$ $(i=0,1,\dots),$\quad
and all other products of basis elements zero.
\end{minipage}\end{equation}

Clearly, for each $i\geq 0,$ this algebra has a homomorphic image
$A_i$ in which all $w_j$ with $j\geq i$ are set to zero, and these
images form an inverse system of nilpotent
algebras, whose inverse limit $A$ consists of all
formal infinite sums $\alpha\,x + \sum_{i=0}^\infty\ \beta_i\,w_i$
$(\alpha,\,\beta_i\in k).$

The ideal $(w_0)$ of $A$ is easily shown to consist of the
finite sums $\beta_0\,w_0+\dots+\beta_n\,w_n.$
In particular, it does not contain the element
$y=\sum_{i=0}^\infty w_i=(1-l_x)^{-1} w_0,$
which satisfies $y-xy=w_0.$
This gives the first assertion;
the remaining assertions follow immediately.
\end{proof}

Still more striking is

\begin{example}\label{E.y-yy}
There exists a pro-nilpotent algebra $A$ over a field $k$
having an element $y$ such that $y\notin (y-y^2).$

Hence in $B=A/(y-y^2),$ the element $y$ spans
an idempotent $\!1\!$-dimensional \textup{(}associative!\textup{)}
subalgebra, in contrast to Lemma~\ref{L.idpt}\textup{(ii)}.
\end{example}

\begin{proof}[Construction and proof]
This time, we start with an algebra having basis
$\{w_0,\,w_1,\dots\},$ and multiplication given by
\begin{equation}\begin{minipage}[c]{35pc}\label{d.xwizi}
$w_i\,w_i\ =\ w_{i+1}$ $(i=0,1,\dots),$\quad
and all other products of basis elements zero.
\end{minipage}\end{equation}

We again get nilpotent homomorphic images $A_i$ on
setting $w_j$ equal to zero for all $j\geq i.$
The inverse limit
$A$ of these algebras consists of all formal infinite sums
\begin{equation}\begin{minipage}[c]{35pc}\label{d.sum}
$\sum_{i=0}^n\,\alpha_i\,w_i\quad (\alpha_i\in k).$
\end{minipage}\end{equation}
Again, it is not hard to see that
\begin{equation}\begin{minipage}[c]{35pc}\label{d.w0z0}
the ideal $(w_0)$ of $A$ consists of all
finite sums $\alpha_0\,w_0+\dots+\alpha_n\,w_n.$
\end{minipage}\end{equation}

Again let $y=\sum_{i=0}^\infty\ w_i.$
We find that $y-y^2=w_0,$
so $(y-y^2)$ is $(w_0),$ the ideal described in~(\ref{d.w0z0}),
which clearly does not contain $y.$
This proves the main assertion; the final statement again follows.
\end{proof}

Even for associative algebras, Lemma~\ref{L.idpt}(ii)
only excludes nonzero {\em finitely generated} idempotent subalgebras.
An easy example of a radical associative algebra $B$ with a
non-finitely-generated idempotent subalgebra $S$
is gotten by taking for both $B$ and $S$ the maximal ideal of any
nondiscrete valuation ring.
It is harder to get an example with $B$ a homomorphic image of a
pro-nilpotent algebra, but the following celebrated construction of
S\c{a}siada and Cohn~\cite{ES+PMC} has that property.
(For parallelism with the other examples of this section, I have
interchanged below the use of the symbols $x$ and $y$
in~\cite{ES+PMC}.)
Note that the ideal $(y)$ in the statement, though generated by
a single element as a $\!2\!$-sided ideal of $B,$ may (and must,
by Lemma~\ref{L.idpt}(i)-(ii)) be infinitely generated both as a left
ideal of $B$ and as a subalgebra.

\begin{example}[{S\c{a}siada and Cohn \cite{ES+PMC}}]\label{E.ES+PMC}
If $k$ is a field, then in the pro-nilpotent
associative algebra $A=[k\<]\langle\langle x,\,y\rangle\rangle,$
one has $y\notin (y-x\,y^2\,x).$

Thus, in $B=A/(y-x\,y^2\,x),$ the ideal $(y)$
satisfies $B(y)=(y)$ as a left ideal, and $(y)^2=(y)$ as a subalgebra.
Moreover, if $U$ is a maximal ideal of $B$ not containing $y,$
then in $B'=B/U,$ the subalgebra $(y)$ is simple.
Thus, the ideal $(y)$ of $B'$ is a simple Jacobson radical algebra.
\end{example}

\begin{proof}[Sketch of proof.]
The proof that $y\notin (y-x\,y^2\,x)$ occupies most of the
five pages of~\cite{ES+PMC}, and I will not discuss it here.

That relation clearly yields the asserted equalities $B(y)=(y)$
(which we also had in Example~\ref{E.y-xyz})
and $(y)^2=(y)$ (which we did not have in our previous
associative examples).

Let us now prove the simplicity, as a ring, of the
ideal $(y)$ of $B'=B/U,$ though this was also done in~\cite{ES+PMC}.
By maximality of $U,$ $(y)$ contains no proper nonzero $\!B'\!$-ideal;
suppose, however, that in
contradiction to our desired conclusion, it contained
a proper nonzero $\!(y)\!$-ideal $V.$
If $(y)V=\{0\},$ then the right annihilator
of $(y)$ in $(y),$ which is clearly an ideal of $B'$ properly
contained in $(y),$ is nonzero, a contradiction.
So $(y)V\neq\{0\}.$
Knowing this, we see in
turn that if $(y)V(y)=\{0\},$ then the left annihilator
of $(y)$ in $(y)$ gives the same contradiction.
Hence $(y)V(y)\neq\{0\}.$
But this, too, is an ideal of $B'$ contained in $V,$
hence properly contained in
$(y),$ a final contradiction that completes the proof.

In the final sentence of the lemma
(which was the goal of \cite{ES+PMC}), radicality
holds because any ideal of a radical ring is radical.
\end{proof}

Examples~\ref{E.y-xyz}-\ref{E.ES+PMC} can all be thought of as
illustrating, in one way or another, the fact
that the conclusion of Theorem~\ref{T.main}(ii) can fail if one
deletes the hypothesis that the underlying $\!k\!$-module of $B$
is Hopfian.
We end this section with a much easier example showing that even if
that module is Hopfian, this is not enough to give the full assertion
of part~(iii).

\begin{example}\label{E.hopf}
There exist a commutative ring $k$
and a pro-nilpotent commutative associative $\!k\!$-algebra
$A$ which is Hopfian as a $\!k\!$-module, but not nilpotent
as an algebra.
\end{example}

\begin{proof}[Construction and proof]
Let $k$ be a complete discrete valuation ring, with
maximal ideal $(p),$ and consider the inverse system of nilpotent
algebras $A_i=(p)/(p^i)$ $(i\geq 1)$ with the obvious surjective
connecting homomorphisms.
Because $k$ is complete, the inverse limit $A$ of this
system is isomorphic as a $\!k\!$-algebra
to the ideal $(p)\subseteq k,$ which is free of rank~$1$ as a
$\!k\!$-module, hence is Hopfian, but is not nilpotent as an algebra.
\end{proof}

(If we had left out the assumption that $k$ was complete, our $A$
would have been the maximal ideal $p\,\hat{k}$ of the completion
$\hat{k}$ of $k.$
In that situation, $\hat{k},$ and hence that ideal, would again be
Hopfian as a $\!k\!$-module, but for less obvious reasons.)

\section{A chain of conditions.}\label{S.chain}
The proof of Theorem~\ref{T.main} involves a chain of
conditions on an algebra $A:$
\vspace{.5em}
\begin{equation}\label{d.nilp}
\mbox{$A$ is nilpotent; equivalently, $M(A)$ is nilpotent.}
\end{equation}
$$\Downarrow$$
\begin{equation}\begin{minipage}[c]{18pc}\label{d.rad}
$M(A)$ is Jacobson radical; equivalently, for every
$u\in M(A),$ $1+u$ is invertible in $1+M(A).$
\end{minipage}\end{equation}
$$\Downarrow$$
\begin{equation}\label{d.M_in_rad}
\mbox{$M(A)$ is contained in a Jacobson radical subalgebra of $\En(A).$}
\end{equation}
$$\Downarrow$$
\begin{equation}\begin{minipage}[c]{20pc}\label{d.mxsurj}
For every $n>0$ and $u\in\mathrm{Mat}_n(M(A)),$ $1+u$
is \mbox{surjective} as a map on the direct sum of $n$ copies of~$A.$
\end{minipage}\end{equation}
$$\Downarrow$$
\begin{equation}\label{d.surj}
\mbox{For every $u\in M(A),$ $1+u$ is surjective as a map $A\to A.$}
\end{equation}
\vspace{.5em}

We note some quick examples showing that the first four of these
conditions are distinct:

In any commutative local integral domain which is not
a field, the maximal ideal $A$
is a radical subalgebra, and satisfies $M(A)\cong A,$ hence
$A$ satisfies~(\ref{d.rad}), but not~(\ref{d.nilp}).

Amplifying the comment preceding Example~\ref{E.hopf}, we note that
for the algebras $A$ of Examples~\ref{E.y-xyz}-\ref{E.ES+PMC} the
inclusion $M(A)\subseteq\lm_I M(A_i)$
yields~(\ref{d.M_in_rad}), but that these
algebras cannot satisfy~(\ref{d.rad}), since they have
homomorphic images $B$ on which certain operators $1+u$ $(u\in M(B))$
are non-invertible.

The algebras $B$ of those same examples
satisfy~(\ref{d.mxsurj}) by Theorem~\ref{T.main}(i),
but they have elements $u\in M(B)$ such that $1+u$ is not injective,
hence not invertible, so they do not satisfy~(\ref{d.M_in_rad}).

I do not know whether there are algebras satisfying~(\ref{d.surj})
but not~(\ref{d.mxsurj}).
\vspace{.5em}

Conditions~(\ref{d.nilp}), (\ref{d.rad}), (\ref{d.mxsurj})
and~(\ref{d.surj}) clearly carry over to homomorphic images; but
Examples~\ref{E.y-xyz}-\ref{E.ES+PMC} show
that~(\ref{d.M_in_rad}) does not;
though Proposition~\ref{P.Hopfian} shows that it does
when the image algebra is Hopfian as a $\!k\!$-module.

Condition~(\ref{d.nilp}) carries over to subalgebras,
but none of the others do.
E.g., in a discrete
valuation ring, such as the localization $\mathbb{Z}_{(p)}$
of $\mathbb{Z}$ at a prime $p$
(notation unrelated to the $A_{(n)}$ of \S\ref{S.np}!),
or a formal power series algebra $k[[t]]$ over a field
$k,$ the maximal ideal
(in these cases, $p\,\mathbb{Z}_{(p)},$ respectively $[k\<][[t]]),$
regarded as an algebra, satisfies~(\ref{d.rad}), and
hence~(\ref{d.M_in_rad})-(\ref{d.surj});
but in these two examples, the $\!\mathbb{Z}\!$-subalgebra
$p\,\mathbb{Z}\subseteq p\,\mathbb{Z}_{(p)},$ respectively
the $\!k[t]\!$-subalgebra $[k\<][t]\subseteq [k\<][[t]],$ fails to
satisfy~(\ref{d.surj}), hence likewise~(\ref{d.rad})-(\ref{d.mxsurj}).

What about inverse limits; say with respect to inverse systems
where the $p_i$ are surjective?
Examples~\ref{E.y-xyz}-\ref{E.ES+PMC} show that~(\ref{d.nilp})
and~(\ref{d.rad}) fail to carry over to these.
Probably~(\ref{d.mxsurj}) and~(\ref{d.surj}) do not carry over either
-- those conditions make the maps $1+u$ surjective
but do not make inverse images of elements under those maps
unique, and this leads to no way of lifting such inverse images to the
inverse limit algebra
(unless the indexing set $I$ has countable cofinality).
Condition~(\ref{d.M_in_rad}) seemed the most likely to
yield a positive result for surjective inverse limits.
If for each $i$ we let $N(A_i)$ denote the least radical
subalgebra of $\En(A_i)$ containing $M(A_i),$ we might hope to use
the fact that an inverse limit of radical algebras is radical.
Unfortunately, it does not appear that the connecting maps
$f_{ji}$ will induce maps $N(A_i)\to N(A_j):$ without a Hopfian
condition on $A_i,$ there is no reason why the quasiinverse
of a map carrying $\ker(f_{ji})$ into itself should
likewise carry $\ker(f_{ji})$ into itself.
So we have no positive results for any of our conditions.
\vspace{.5em}

We remark that variants of~(\ref{d.mxsurj}) and~(\ref{d.surj})
in which ``surjective'' is replaced by ``injective'' or by
``bijective'' might also be of interest.

\section{A Nakayama-like property.}\label{S.Nak-like}

In the proof of Theorem~\ref{T.main}, we obtained statement~(iii)
from statement~(ii) essentially by showing that
for an algebra of finite length as a
module, condition~(\ref{d.M_in_rad}) implies~(\ref{d.nilp}).
On the other hand, we did not obtain~(ii) directly from~(i)
-- I do not know whether for algebras that are Hopfian as
modules,~(\ref{d.mxsurj}) or~(\ref{d.surj}) implies~(\ref{d.M_in_rad}).
(If the implication requires the stronger statement~(\ref{d.mxsurj}),
then it probably needs not only the hypothesis that $A$ is
Hopfian, but that all finite direct sums $\bigoplus_n A$ are Hopfian.)
If $A$ is Hopfian and satisfies~(\ref{d.surj}),
the maps $1+u$ $(u\in M(A))$ are invertible, hence such $u$
are quasiinvertible in $\En(A);$ but if
we try to extend $M(A)$ within $\En(A)$ to a subalgebra
closed under quasiinverses, it could happen that
the new quasiinvertible elements we introduce (quasiinverses of
old elements) will yield sums or products that fail to be
quasiinvertible (like $l_x r_z$ and $l_x+r_x$
in Examples~\ref{E.y-xyz} and~\ref{E.y-xy+yx});
so we could fail to get~(\ref{d.M_in_rad}).

However, whether or not we can deduce~(\ref{d.M_in_rad}),
if~(\ref{d.mxsurj}) holds and the
modules $\bigoplus_n A$ are Hopfian, then
$M(A)$ will behave somewhat like a radical algebra,
in that it will satisfy a version of Nakayama's lemma.
Let us formulate this result with $M(A)$ and $A$ generalized
to an arbitrary ring and module.

\begin{lemma}\label{L.Nak-like}
Let $R$ be a nonunital associative ring,
$A$ a left $\!R\!$-module, and $n$ a positive integer.
Suppose that for every $u\in\mathrm{Mat}_n(R),$ the
element $1+u$ acts in a one-to-one fashion on $\bigoplus_n A.$

Then for any $\!n\!$-generator $\!R\!$-submodule
$C$ of $A$ \textup{(}or more generally,
for any $\!n\!$-generator $\!R\!$-submodule $C$ of a direct
product of copies of $A),$ one has $R\,C=C\implies C=\{0\}.$
\end{lemma}

\begin{proof}
The case where $C$ is a submodule of a direct product
reduces immediately to the case $C\subseteq A$ by projecting
onto any coordinate where some member of $C$ has nonzero component.
So assume $C$ is a submodule of $A,$ generated by $x_1,\dots,x_n.$

The relation $R\,C=R$ means that each $x_m$ can be written
as an $\!R\!$-linear combination of itself and the others;
which says that if we let $x=(x_1,\dots,x_n)\in \bigoplus_n A,$
then for some $u\in\mathrm{Mat}_n(R),$ we have $x=ux.$
But this says that $1-u$ annihilates $x,$ contradicting our hypothesis.
\end{proof}

Using the above result, we can
get part~(iii) of Theorem~\ref{T.main} directly from part~(i),
via the observation

\begin{corollary}\label{C.i>iii}
For $A$ an algebra of finite length as a
$\!k\!$-module,\textup{~(\ref{d.mxsurj})}
implies\textup{~(\ref{d.nilp})}.
\textup{(}So for such
$A,$\textup{~(\ref{d.nilp})-(\ref{d.mxsurj})} are equivalent.\textup{)}
\end{corollary}

\begin{proof}
If~(\ref{d.nilp}) fails, then the decreasing chain of submodules
$M(A)^d(A)$ of $A$ $(d=0,1,\dots)$ never becomes zero;
but by the finite length assumption, it must stabilize.
Thus, say $\{0\}\neq C=M(A)^d(A)$ satisfies $C\ =\ M(A)(C).$
By our finite length hypothesis, $C$
is finitely generated, say by $n$ elements.

Since a module of finite length is
Hopfian, condition~(\ref{d.mxsurj}) says
that all elements $1+u$ $(u\in\mathrm{Mat}_n(M(A)))$
act invertibly on $\bigoplus_n A,$ hence in a one-to-one fashion;
so Lemma~\ref{L.Nak-like} says that $C=\{0\},$ a contradiction.
\end{proof}

The next corollary to Lemma~\ref{L.Nak-like} shows that
Theorem~\ref{T.main}(ii) has concrete consequences for the
structure of Hopfian homomorphic images of pro-nilpotent algebras.
(Note that the hypothesis of finite generation as an ideal
is weaker than finite generation as a one-sided ideal.)

\begin{corollary}\label{C.not_idpt}
A nonzero $\!k\!$-algebra $A$ satisfying\textup{~(\ref{d.mxsurj})}
\textup{(}in particular, a nonzero homomorphic image of
a pro-nilpotent algebra\textup{),} which
has the property that for all $n$
the $\!k\!$-module $\bigoplus_n A$ is Hopfian, cannot be
idempotent as a $\!k\!$-algebra and finitely generated as an ideal.
\end{corollary}

\begin{proof}
For any algebra $A,$ the ideals of $A$ are
its $\!M(A)\!$-submodules, and we see that the conditions
of idempotence as an algebra and finite generation as an ideal say
that $M(A)A=A$ and $A$ is finitely generated as an $\!M(A)\!$-module.
However, the Hopfian condition together
with~(\ref{d.mxsurj}) yield the hypothesis of
Lemma~\ref{L.Nak-like}, implying that $A=\{0\}.$
\end{proof}

\section{Solvable Lie algebras.}\label{S.solv}
The {\em derived series} of an algebra $A$ is the
sequence of subalgebras $A^{(n)}$ $(n=0,1,\dots)$ defined by
\begin{equation}\begin{minipage}[c]{35pc}\label{d.der_series}
$A^{(0)}=A,\qquad A^{(n+1)}= A^{(n)}\,A^{(n)}.$
\end{minipage}\end{equation}
This concept is standard in the theory of Lie algebras (where
the $A^{(n)}$ are in fact ideals).
It is less so for general nonassociative algebras,
but is introduced in that context in \cite[p.17]{Schafer}.

An algebra $A$ is called {\em solvable} if $A^{(n)}=\{0\}$ for
some $n\geq 0.$
It is easy to see that $A^{(n)}\subseteq A_{(2^n)}$ $(n=0,1,\dots\,),$
so every nilpotent algebra is solvable; but the converse is
not true, as shown by the $\!2\!$-dimensional Lie algebra
with basis $\{x,\,y\}$ and multiplication $[x,\,y]=y.$

There is a special characterization of solvability of
Lie algebras in the classical case:
\begin{equation}\begin{minipage}[c]{35pc}\label{d.solv}
\cite[Corollary~1 to Theorem~13, p.51.]{NJ_Lie}
If $A$ is a finite-dimensional Lie algebra over a field
of characteristic~$0,$ then $A$ is solvable
if and only if its commutator ideal $A^{(1)}=[A,\,A]$ is nilpotent.
\end{minipage}\end{equation}

Nazih Nahlus has pointed out that using
this fact, one gets as a consequence of our main theorem the
following result, which he had conjectured some years ago.

\begin{corollary}[{to Theorem~\ref{T.main}(iii).~ N.\,Nahlus (personal communication)}]\label{C.solv}
Let $A$ be an inverse limit of finite-dimensional solvable Lie
algebras $A_i$ over a field $k$ of characteristic~$0.$
Then any finite-dimensional homomorphic image $B$ of $A$ is solvable.
\end{corollary}

\begin{proof}
By~(\ref{d.solv}), the commutator ideals of the $A_i$
form an inverse system of nilpotent algebras.
The inverse limit $A^*\subseteq A$ of this system contains
all brackets of elements of $A;$ so
when we map  $A$ homomorphically onto a finite-dimensional algebra $B,$
the image of $A^*$ contains all brackets of elements of $B.$
Theorem~\ref{T.main}(iii) tells us that that image
is nilpotent, so $B$ is solvable.
\end{proof}

However, there are both infinite-dimensional Lie algebras $A$
in characteristic $0,$ and finite-dimensional Lie algebras $A$
in positive characteristic, which are solvable,
but for which $A^{(1)}$ is not nilpotent.

An example of the former is given by the vector space $A$ of operators
on $\mathbb{R}[x]$ spanned by the operators $X^n$ of multiplication
by $x^n$ $(n=0,1,\dots),$ together with the operator $D=d/dx,$
and the composite operator $XD=x\,d/dx.$
Indeed, one verifies that $A$
is closed under commutator brackets, hence forms a Lie algebra
(the semidirect product of the $\!2\!$-dimensional
Lie algebra $L$ spanned by $\{D,\,XD\},$
and the $\!L\!$-module $\mathbb{R}[x]).$
One finds that the subalgebra $A^{(1)}=[A,A]$ is spanned
by all the above operators except $XD.$
(In particular, $[D,\,XD]=D$ does appear.)
This subalgebra is not nilpotent, since $[D,X^n]=nX^{n-1},$
so that there are elements which can be bracketed with $D$
arbitrarily many times before going to zero.
At the next step, however,
$A^{(2)}=[A^{(1)},A^{(1)}]$ is spanned by the operators $X^n$ only,
and hence has zero bracket operation, so $A^{(3)}=\{0\},$
showing that $A$ is solvable, even though $A^{(1)}$ is not nilpotent.

To get a finite-dimensional example in positive characteristic, let
us first note a variant of the above characteristic~$0$ example.
Consider the ring of functions $\mathbb{R}[x,\,e^x],$ and the space
of operators on that ring spanned by $D$ and $XD$ as above, together
with (rather than the operators $X^n)$ the operators
$X^n\,Y$ $(n\geq 0),$
where $Y$ is the operator of multiplication by $e^x.$
Again, one verifies that this is closed under commutator
brackets, and so gives a Lie algebra $A$
(the semidirect product of $L$ as above and
the $\!L\!$-module $\mathbb{R}[x]\,e^x).$
Where in the preceding example, the infinite-dimensionality of
$\{X^0,\,X^1,\,X^2,\dots\}$ was involved in establishing the
non-nilpotence of $A^{(1)},$ here non-nilpotence follows from the
single relation
\begin{equation}\begin{minipage}[c]{35pc}\label{d.DX0Y}
$[D,\,X^0\,Y]=X^0\,Y.$
\end{minipage}\end{equation}
This does not allow us to pass to a finite-dimensional
subalgebra with the desired properties, because the iterated
action of $XD$ on $X^0\,Y$ brings in all the $X^n\,Y.$
However, one finds that the structure constants of this
Lie algebra with respect to our basis are integers, and
that when one reduces them modulo a prime $p,$ then the span
of $\{X^p\,Y,\,X^{p+1}Y,\,\dots\}$ becomes an ideal.
(Key calculation: in the original algebra,
$[D,\,X^pY]=pX^{p-1}Y+X^pY,$ and modulo $p,$
the first term of that expression vanishes.)
The factor-algebra by that ideal
is thus a $\!(p+2)\!$-dimensional Lie algebra $B,$ and the
relation~(\ref{d.DX0Y}) shows that $B^{(1)}$ is still not nilpotent.
However, we find that $B^{(1)}$ again loses the operator $XD,$
that $B^{(2)}$ likewise loses $D,$ hence has zero bracket operation,
so that again $B^{(3)}=\{0\}$ and $B$ is solvable.
Further examples in prime characteristic may
be found in~\cite{KB+DAT}.

So if Corollary~\ref{C.solv} is to be extended
to positive characteristic, or
to inverse limits of not necessarily finite-dimensional Lie algebras,
or to non-Lie algebras, a very different proof will be needed.

One can, of course, generalize that corollary and its
present proof by strengthening
the hypothesis to {\em assume} $A$ is an inverse limit
of Lie algebras for which $A^{(1)}$ is nilpotent.
Indeed, one can generalize the resulting statement to arbitrary
algebras, replacing solvability by any condition specifying that
the values of a given family of algebra terms should generate
a nilpotent subalgebra.

\section{Possible variants of our main theorem.}\label{S.variants}

Let us look at a few ways Theorem~\ref{T.main}
can, or might, be generalized.

We start with one that, as a generalization, proves disappointing;
but which shows that our present
Theorem~\ref{T.main} is stronger than we realized.

\subsection{General limits.}\label{S2.genlim}
Recall that the concept of the inverse limit of an inversely directed
system of algebraic structures is a case of the more general
category-theoretic notion of the ``limit'' of a functor
\cite[\S\S7.6]{245} \cite[\S III.4]{CW}, other important examples
of which are the fixed-point algebra of a group acting
on an algebra, and the equalizer of a pair of algebra homomorphisms.
If one examines the proof of Theorem~\ref{T.main}, one sees no
reason why it should not work for limits in this general sense.
It does -- but that extension gives nothing new:

\begin{lemma}\label{L.gen_lims}
For a $\!k\!$-algebra $A,$ the following conditions are equivalent.

\textup{(i)}~ $A$ can be written as the limit of a system of
nilpotent $\!k\!$-algebras indexed by a small category.

\textup{(ii)}~ $A$ is pro-nilpotent, i.e.,
can be written as an inverse limit of an
inversely directed system of nilpotent $\!k\!$-algebras.
\end{lemma}

\begin{proof}[Sketch of proof]
Clearly, (ii)$\implies$(i).

Conversely, suppose $F:\mathbf{C}\to\mathbf{Alg}_k$ is a functor
from a small category $\mathbf{C}$ to the category $\mathbf{Alg}_k$
of not-necessarily-associative $\!k\!$-algebras, such
that for all $X\in\mathrm{Ob}(\mathbf{C}),$ $F(X)$ is nilpotent.

Let $I$ be the partially ordered set of finite subsets of
$\mathrm{Ob}(\mathbf{C}),$ ordered by reverse inclusion;
clearly, $I$ is inversely directed.
For each $i\in I,$ let $\mathbf{C}_i$ be the full subcategory
of $\mathbf{C}$ with object-set $i,$ and let
$A_i=\lm\,(F\,|\,\mathbf{C}_i),$ where $F\,|\,\mathbf{C}_i$ denotes
the restriction of $F$ to $\mathbf{C}_i.$

Each $A_i$ is a subalgebra of the finite product
$\prod_{X\in i} F(X),$ and the class of nilpotent
algebras is closed under finite products and
subalgebras, hence each $A_i$ is nilpotent.
Given $i\leq j\in I,$ which by our ordering of $I$
means $i\supseteq j,$ the inclusion
$\mathbf{C}_j\subseteq\mathbf{C}_i,$ induces a restriction
homomorphism $A_i\to A_j.$
It is straightforward to verify that $\lm F=\lm_I A_i,$ yielding~(ii).
\end{proof}

\subsection{Variant sorts of nilpotence.}\label{S2.varnilp}

Within the multiplier algebra $M(A)$ of an algebra $A,$ we may
look at the subalgebra $M_l(A)$ generated by the left multiplication
operators $l_x,$ and the subalgebra $M_r(A)$ generated by
the right multiplication operators $r_x.$

If $A$ is associative, these give nothing very new:
$M_l(A)$ is isomorphic to the factor-algebra of $A$
by its left annihilator ideal $\{x\in A\mid xA=\{0\}\},$
and $M_r(A)$ is antiisomorphic to the factor-algebra
of $A$ by the analogous right annihilator ideal;
so each is nilpotent if and only if $A$ is.
If, rather, $A$ is anticommutative (e.g., is a Lie algebra)
or is commutative (e.g., is a Jordan algebra), then $M_l(A)$
and $M_r(A)$ coincide with $M(A).$

But for a general nonassociative algebra $A,$ these two
subalgebras of $M(A)$ can look very different.
For instance, for the algebra with
multiplication~(\ref{d.xwi...}), it is easy to see that $(AA)A=\{0\},$
so that $M_r(A)^2=\{0\},$ but that $M_l(A)^n\neq \{0\}$ for all $n.$

The conditions $(\exists\,n)\,M_l(A)^n=\{0\}$
and $(\exists\,n)\,M_r(A)^n=\{0\}$ are known
as {\em left} nilpotence and {\em right} nilpotence \cite{AMSl}.
An algebra can be both left and right nilpotent
without being nilpotent, as shown by the algebra with
basis $x, w_0, w_1,\dots\,,$ and multiplication
\begin{equation}\begin{minipage}[c]{35pc}\label{d.xw2i2i+1}
$x\,w_{2i}\ =\ w_{2i+1},\quad w_{2i+1}\,x\ =\ w_{2i+2},$\quad
all other products of basis elements being zero.
\end{minipage}\end{equation}

The development of Theorem~\ref{T.main} goes over,
with no change, with the condition of left nilpotence or of
right nilpotence in place of the condition of nilpotence!
It is not clear to me what the most useful common generalization
of these various sorts of nilpotence is, so I leave it to the experts
in nonassociative rings to develop that observation further.

Let us record a few other versions of nilpotence,
corresponding to still other subalgebras of $M(A).$

Given any $\alpha,\beta\in k,$
one can define a new multiplication on any $\!k\!$-algebra $A$ by
\begin{equation}\begin{minipage}[c]{35pc}\label{d.x*y}
$x*y\ =\ \alpha\,x\,y+\beta\,y\,x,$
\end{minipage}\end{equation}
(from which the original multiplication is recoverable by a
transformation of the same form if $\alpha^2-\beta^2$
is invertible in $k).$
Left nilpotence of this operation is a property of $A$ that is not,
in general,
equivalent to either nilpotence, left nilpotence, or right nilpotence
of the original operation; but any results on
left nilpotence of a general algebra will necessarily apply to
left nilpotence of this operation.

Recall next that for any algebra $A$ one can define the family of
{\em associator} operations by
\begin{equation}\begin{minipage}[c]{35pc}\label{d.associator}
$a_{x,z}(y)\ =\ x(yz)-(xy)z\qquad (x,y,z\in A).$
\end{minipage}\end{equation}
Hence we may consider the subalgebra $M_a(A)\subseteq M(A)$ generated by
all these maps, and study algebras $A$ for which $M_a(A)$ is nilpotent.
Does the fact that the generating set of maps $\{a_{x,z}\mid x,z\in A\}$
is not a linear image of $A$ but a bilinear image of $A\times A$
affect the usefulness of this construction?
I don't know.

Finally, note that to every finite binary tree with $n$ leaves, one
can associate a way of bracketing $n$ symbols, and hence a way of
associating to every algebra $A$ a derived $\!n\!$-ary operation.
Various nilpotence-like conditions can be expressed conveniently in
terms of this formalism.
Thus, an algebra $A$ is left nilpotent if and only if for some $n,$ the
$\!n+1\!$-ary operation induced by the length-$\!n\!$ right-branching
chain is zero on $A;$ right nilpotent,
likewise, if and only if for some $n,$ the
operation induced by the length-$\!n\!$ left-branching chain is zero.
(Here we call a tree a ``chain'' if after pruning all leaves, it has
the form usually called a chain.)
An algebra $A$ is nilpotent if and only if for some $n,$ the operations
induced by all length-$\!n\!$ chains are zero; equivalently, if
and only if for
some $n'$ the operations induced by all trees with $n'$ leaves are zero.
An algebra is solvable if and only if for some $n$ the operation
induced by the full depth-$\!n\!$ binary tree (with $2^{n+1}-1$
nodes) is zero.
One might put these conditions into a general framework by associating
conditions on algebras to appropriate filters of subsets of the
set of finite binary trees.

\subsection{What about restricted Lie algebras?}\label{S2.p-Lie}
Over a field $k$ of characteristic $p>0,$ a more useful concept than
that of a Lie algebra is that of a {\em restricted} Lie
algebra or {\em $\!p$-Lie algebra}\/: a Lie algebra given with
an additional operation, $x\mapsto x^{(p)},$ satisfying certain
identities which, in associative $\!k\!$-algebras,
relate the $\!p$-th power map with the
$\!k\!$-module structure and commutator brackets.
Though $\!p$-Lie algebras are
not algebras in the sense of this note, it would, of course,
be of interest to know whether versions of our
results hold for these objects.

\subsection{What about groups?}\label{S2.groups}
The relation between nilpotence of Lie algebras over
$\mathbb{R}$ and $\mathbb{C},$ and
nilpotence in the sense of group theory of the corresponding
Lie groups, makes it natural to ask whether the methods and results
of this note have analogs for groups $G$ (not necessarily Lie).

In view of the way the brackets of a Lie algebra are related
to the group operation, the natural analogs of the maps
$r_x$ and $l_x$ in the above development would seem to be
the commutator maps $c_x(y)=x^{-1}y^{-1}x\,y.$
The analog of ``quasiinvertibility'' for a map $u:G\to G$ might
be invertibility of the set-map $g\mapsto g\,u(g)$ of $G$ to itself.
But it is not clear under what operations it would be natural to close
the set of commutator maps to form the analog of $M(A),$ and whether
this (or any method) will lead to an analog of Theorem~\ref{T.main}.

\section{Questions.}\label{S.questions}
Topics for further investigation have been noted above.
Here are some more specific questions.

Regarding the chain of conditions in~\S\ref{S.chain}, we ask

\begin{question}\label{Q.chain}
\textup{(i)}~ For $A$ an algebra, is the implication
\textup{(\ref{d.mxsurj})$\!\implies\!$(\ref{d.surj})} reversible?
More generally, if an associative nonunital algebra $R$ has
a module $A$ such that for each $r\in R,$ the operator
$1+r$ is surjective on $A,$ does the action of
$\mathrm{Mat}_n(R)$ on the direct sum of $n$
copies of $A$ have the same property?

\textup{(ii)}~ For $A$ an algebra which
is Hopfian as a $\!k\!$-module, is any of the the implications
\textup{(\ref{d.rad})$\!\implies\!$(\ref{d.M_in_rad})$\!\implies\!$%
(\ref{d.mxsurj})$\!\implies\!$(\ref{d.surj})} reversible?
\end{question}

Examples~\ref{E.y-yy} and~\ref{E.ES+PMC} show that a
homomorphic image of a pro-nilpotent algebra can contain a simple
subalgebra, and so, in particular, an idempotent subalgebra.
This leaves open the question

\begin{question}\label{Q.idpt}
Can a nonzero homomorphic image $B$ of a pro-nilpotent algebra $A$
over a field $k$ be idempotent?
Simple?
If so, can this happen when our algebras are associative?
\end{question}

Of course, by Theorem~\ref{T.main}(iii), such a $B$ cannot
be finite-dimensional and by Lemma~\ref{L.idpt}(i)-(ii), if our
algebras are associative, $B$ cannot be finitely generated.
For the case where $k$ is not, as assumed above, a field,
Corollary~\ref{C.not_idpt} gives a somewhat weaker restriction.

We have seen ways in which Lie algebras behave like
associative algebras (Lemma~\ref{L.[]=()}), and ways
in which they differ (the contrast between
Lemma~\ref{L.idpt}(i) and Example~\ref{E.y-xy+yx}).
The next question notes some cases where it isn't
clear on which side of the fence Lie algebras will fall.

\begin{question}\label{Q.Lie_idpt}
Can a homomorphic image $B$ of a pro-nilpotent Lie algebra
have a nonzero finitely generated idempotent subalgebra?

If so, can it have a nonzero finitely generated simple subalgebra?

If so, can such a subalgebra be finite-dimensional?
\end{question}

(A curious difference between the behaviors of Lie and
associative algebras is noted in \cite[Example~25.49]{coalg}, where
it is observed that a topological Lie algebra (over a field) with
a linearly compact topology need not be an inverse limit of
finite-dimensional Lie algebras.
The example is the Lie algebra spanned
by $\mathbb{R}[x]$ and $d/dx.$
Under the duality between vector spaces and linearly compact
vector spaces, this shows that the ``Fundamental
theorem on coalgebras'', a result on coassociative coalgebras,
is not valid for co-Lie-algebras.)

In~\S\ref{S.solv}, where we considered solvable Lie algebras,
we raised

\begin{question}\label{Q.solv}
In Corollary~\ref{C.solv}, is it possible to remove or weaken
\textup{(i)}~the condition that the $A_i$ be finite-dimensional, or
\textup{(ii)}~the condition of characteristic~$0,$ or
\textup{(iii)}~the condition that the algebras be Lie?
\end{question}

In~\cite{prod_Lie1} and~\cite{prod_Lie2}, N.\,Nahlus
and the present author study homomorphic
images of {\em direct product} algebras $\prod_I A_i.$
The form of the results obtained there suggest some possible
strengthenings of Theorem~\ref{T.main}(iii):

\begin{question}\label{Q.cf_prod_Lie}
In Theorem~\ref{T.main}\textup{(iii)}, if $k$ is a field
\textup{(}or perhaps, more restrictively, an infinite field\textup{)},
can the hypothesis that $B$ is finite-dimensional
\textup{(}the form that the finite-length hypothesis takes for vector
spaces\textup{)} be weakened to {\em countable}-dimensional?
\textup{(}Cf.~\cite[Theorem~11]{prod_Lie1},
\cite[Theorem~8]{prod_Lie2}.\textup{)}

For any algebra $B,$ let us
write $Z(B)$ for the ideal $\{b\in B\mid b\,B=B\,b=\{0\}\,\}.$
Then if $k$ is infinite and $\mathrm{card}(I)$ is less than
any uncountable measurable cardinal, can the conclusion of
Theorem~\ref{T.main}\textup{(iii)} be strengthened to say that
the composite map $A\to B\to B/Z(B)$ factors through
one of the projections $p_i:A\to A_i$ \textup{(}equivalently,
is continuous in the pro-discrete topology\textup{)}?
Without those cardinality hypotheses, can we say that
$B/Z(B)$ is a homomorphic image of one of the $A_i$?
\textup{(}Cf.~\cite[Proposition~16]{prod_Lie1}.\textup{)}
\end{question}

The concept of measurable cardinal is reviewed
in~\cite[\S15]{prod_Lie1}.
The need, in the second paragraph of the above question, for
the cardinality conditions and for the denominator ``$Z(B)$''
arises from the need for these same restrictions in~\cite{prod_Lie1}
and~\cite{prod_Lie2}.
Indeed, an infinite direct product of algebras is an
inverse limit of finite subproducts, so counterexamples
to statements for infinite products in the absence of those
restrictions are also counterexamples for inverse limits.

Thinking about the counterexamples in \S\ref{S.eg},
and the differences between the kinds of examples that can exist
for associative and for nonassociative algebras, suggested

\begin{question}\label{Q.assoc_im}
If an associative algebra $B$ can be written as a homomorphic
image of a pro-nilpotent algebra, can it be written as
a homomorphic image of an associative pro-nilpotent algebra?

Same question, with associativity replaced by an arbitrary
identity or set of identities.
\end{question}


\end{document}